%% file: main.tex
\documentclass{article}
\usepackage{amssymb,amsmath,amsthm,bbm} 
\usepackage[margin=1in]{geometry}
\usepackage{tikz,caption,subcaption}
\usepackage{graphicx,ctable,booktabs}
\usepackage[mathscr]{euscript}
\usepackage{enumerate, mathdots}
\usepackage{import,biblatex}
\usepackage{listings}
\import{./}{defs.tex}

\usepackage[utf8]{inputenc}

\newtheorem*{thm}{Theorem}
\numberwithin{equation}{section}
\numberwithin{lem}{section}

\title{Inelastic Particle Clusters from Cumulative Momenta}
\author{Kevin Chien, Aidan Mager, Laurel Safranek, Jackson Zariski}
\date{}

\begin{document}

\maketitle

\begin{abstract}
    We consider a physical system comprising discrete massive particles on the real line whose trajectories interact via perfectly inelastic collision. It turns out that polygons formed in a ``cumulative momentum diagram'' of the initial conditions allow us to easily predict how many particle clusters form as time $t\to\infty$. We explore an application of this to a unit mass system with $\pm 1$ velocities.
\end{abstract}

\section{Introduction}
Inelastic particles, or \textit{sticky particles}, have been studied in relation to the \emph{pressureless Euler system} in  $(0, \infty) \times \RR^d$:
\begin{equation}\label{eq:pressureless}
    \begin{cases}
      \partial _t \rho + div(\rho v) = 0 \\
      \partial _t (\rho v) + div(\rho v \otimes v) = 0
    \end{cases}
\end{equation}
where $\rho(t,x)$ and $v(t,x)$ are the density and velocity fields of gas particles, respectively \cite{stefano2020sticky}. Note that these equations may be derived from the usual Euler equations for ideal compressible fluids by letting the pressure go to zero \cite{brenier}.
This system governs how an ideal pressureless system of gas particles evolves over time. Such a system has been proposed by Zeldovic \cite{zel1970gravitational} as a simplified model for early formations of galaxies, where individual dust particles floating in space may collide and form newer, bigger particles. Bianchini \cite{stefano2020sticky} proved that this system admits a unique sticky particle solution.

In one dimension, solutions to \eqref{eq:pressureless} can be obtained in the limit of a discrete system of sticky particles (see \cite{hynd2020sticky}) often with an interaction potential. In this paper, we focus on such a discrete system with no potential, where particles move in piecewise linear trajectories. Our main result reframes this system in terms of a cumulative momentum diagram, providing a new geometric picture of the situation that simplifies the determination of how many clusters of particles form. As an application, we use this approach to link the problem of determining how often a system of unit mass, unit speed particles on the line form one cluster to a problem concerning simple symmetric random walks. We also provide the code to implement our cumulative momentum approach.

In Section 2, we provide the necessary definitions and main theorem statement. Section 3 contains the proof of our result. Section 4 uses this approach on our toy problem, and Section 5 gives a Python code sample.

\subsection*{Acknowledgements}
We would like to thank Jayadev Athreya for his support and bringing this topic to our attention; Rowan Rowlands for the helpful reference to \cite{ciucu2016short}; as well as the Washington Experimental Math Lab for facilitating this collaboration.

\section{Preliminaries}
\subsection{Inelastic particle systems}
A \textit{particle} $a:=(m,x,v)\in\RR_{>0}\times\RR\times \RR$ is an ordered triple whose elements denote mass, initial position, and initial velocity, respectively. A \textit{system of $n$ particles} denotes an $n$-tuple of particles $(a_i)_{i=1}^n$ satisfying $x_i\neq x_j$ for $i\neq j$, and ordered such that $x_i<x_j$ if and only if $i<j$.

To each particle $a_i$ in our system, we may associate a unique trajectory $\gamma_i:[0,\infty)\to\RR$, which is a continuous, piecewise linear map satisfying the following properties (see Proposition 2.1 in \cite{hynd2020sticky}):
\begin{enumerate}
    \item $\gamma_i(0)=x_i,\, \dot\gamma(0+)=v_i$;\label{init}
    \item If $\gamma_i(s)=\gamma_j(s)$, then $\gamma_i(t)=\gamma_j(t)$ for all $t\geq s$ \textit{(stickiness)}; and \label{stick}
    \item If $\gamma_{i_1}(t)=\cdots=\gamma_{i_\ell}(t)\neq \gamma_{i}(t)$ at some $t>0$, then \[\dot\gamma_{i_j}(t+)=\frac{\sum_{k=1}^\ell m_{i_k}\dot\gamma_{i_k}(t-)}{\sum_{k=1}^\ell m_{i_k}}\]
    for $i_j\in\{i_1,\ldots,i_\ell\}$ \textit{(conservation of momentum)}.\label{cons}
\end{enumerate}

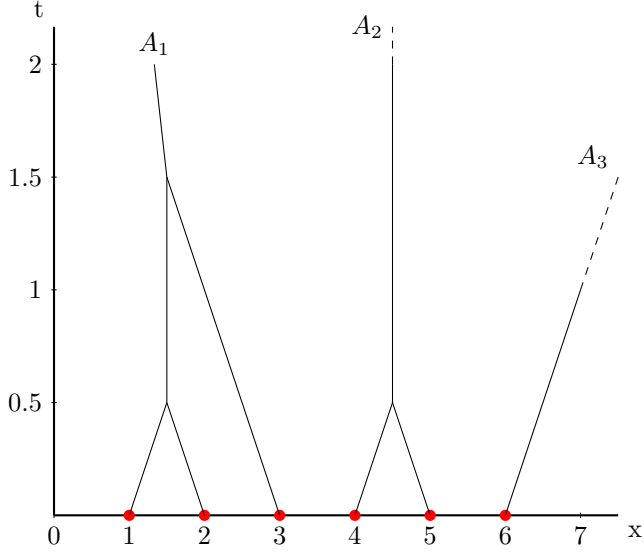
\begin{figure}
\centering
  \begin{tikzpicture}

    \draw[thick,-] (0,0) -- (7.5,0) node[anchor=north west] {x};
    \draw[thick,-] (0,0) -- (0,6.5) node[anchor=south east]{t};
    \foreach \x in {0,1,2,3,4,5,6,7}
    \draw (\x cm,1pt) -- (\x cm,-1pt) node[anchor=north] {$\x$};
    \foreach \x in {1,2,3,4,5,6}
    \draw (\x,0)node[circle,red,fill,inner sep=1.5pt]{};
    \foreach \y in {1.5,3,4.5,6}
    \draw (1pt,\y cm) -- (-1pt,\y cm);
    \draw(-1pt,1.5 cm) node[anchor=east]{$0.5$};
    \draw(-1pt,3 cm) node[anchor=east]{$1$};
    \draw(-1pt,4.5 cm) node[anchor=east]{$1.5$};
    \draw(-1pt,6 cm) node[anchor=east]{$2$};
    \draw (1,0) -- (1.5,1.5) -- (1.5,4.5) --(4/3,6) node[anchor=south]{$A_1$};
    \draw (2,0) -- (1.5,1.5);
    \draw (3,0) -- (1.5,4.5);
    \draw (4,0) -- (4.5,1.5);
    \draw (5,0) -- (4.5,1.5)--(4.5, 6);
    \draw[dashed] (4.5, 6)--(4.5, 6.5)node[anchor= east]{$A_2$};
    \draw (6,0) -- (7,3);
    \draw[dashed] (7,3) -- (7.5,4.5)node[anchor=south east]{$A_3$};
\end{tikzpicture}
\caption{Clusters for unit mass particles with initial velocities $(+1,-1-1,+1,-1,+1)$.}
\label{fig:traj}
\vspace{10pt}
\end{figure} 

In the above, we employ the notation $\dot\gamma(t\pm):=\lim_{h\to0^{\pm}}(\gamma(t+h)-\gamma(t))/h$. We denote $\dot\gamma(t+)$ the velocity of $a_i$ at time $t$: in particular, the velocity exists for all $t\in[0,\infty)$, but may not be continuous.

Observe that the continuity of trajectories and Property \ref{stick} together imply that if particles $a_i,a_j$ collide by time $t$, then all particles $a_k$ for $i\leq k\leq j$ must have collided by time $t$. Also, by expanding the right hand side of the expression in Property \ref{cons}, we can recover the velocity of a particle at any time from the initial conditions of all particles involved in the collision. That is, at a given time $t\in[0,\infty)$, if $J\subset(1,\ldots, n)$ is the subset of indices such that  $\gamma_i(t)=\gamma_j(t)$ for all $j\in J$, then the velocity of $a_i$ is 
\begin{align}
\dot\gamma_{i}(t+)=\frac{\sum_{j\in J} m_{j}v_j}{\sum_{j\in J} m_j}.\label{velo}
\end{align}

Using these trajectories to define an equivalence relation, we can partition our system into \textit{clusters}: two particles $a_i,a_j$ are in the same cluster $A_k$ if and only if $\lim_{t\to\infty}|\gamma_i(t)-\gamma_j(t)|=0$. By our observation above, we may order our clusters $(A_i)_{i=1}^\ell$ such that $i<j$ if and only if the initial position of any particle in $A_i$ is less than the initial position of any particle in $A_j$ (see Figure  \ref{fig:traj}). The velocity of a cluster is the velocity of any constituent particle, for $t$ large enough such that all particles in the cluster share the same trajectory.

\begin{figure}
\centering
  \begin{tikzpicture}
    \draw[thick,-] (0,0) -- (6,0);
    \draw[thick,-] (0,-3.65) -- (0,3.65);
    \foreach \x in {1,2,3,4,5,6}
    \draw (\x cm,1pt) -- (\x cm,-1pt) node[anchor=north] {$\x$};
    \foreach \y in {-3,-1.5,0,1.5,3}
    \draw (1pt,\y cm) -- (-1pt,\y cm);
    \draw(-1pt,-3 cm) node[anchor=east]{$-2$};
    \draw(-1pt,-1.5 cm) node[anchor=east]{$-1$};
    \draw(-1pt,0 cm) node[anchor=east]{$0$};
    \draw(-1pt,1.5 cm) node[anchor=east]{$1$};
    \draw(-1pt,3 cm) node[anchor=east]{$2$};
    
    \draw (0,0)--(1,1.5)node[circle,blue,fill,inner sep=1.5pt]{} -- (2,0)node[circle,blue,fill,inner sep=1.5pt]{} --  (3,-1.5)node[circle,blue,fill,inner sep=1.5pt]{} --
    (4,0)node[circle,blue,fill,inner sep=1.5pt]{} --
    (5,-1.5)node[circle,blue,fill,inner sep=1.5pt]{};
    \draw (6,0)node[circle,blue,fill,inner sep=1.5pt]{};
    \draw[red]  (0,0)--(3,-1.5)--(5,-1.5) -- (6,0);
    \draw (1,.5)node[]{$T_1$};
    \draw (4,-1.25)node[]{$T_2$};
    \draw (5.55,-1.25)node[]{};
    \draw[red] (5.75,-.5)node[anchor=north west]{$\Gamma(r)$};
    \draw (1.2,1.5)node[anchor=north west]{$\Gamma(f)$};
    \draw (5,-1)node[anchor=south]{$T_3$};
\end{tikzpicture}
  \caption{The cumulative momentum diagram for the same system.} 
  \label{fig:sub2}
  \vspace{10pt}
\end{figure}
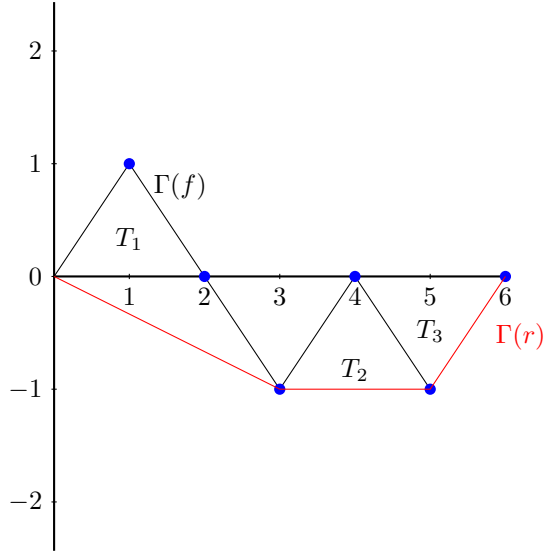
\subsection{Cumulative momentum diagrams and polygons}

Given a system $(a_i)_{i=1}^n$, we now give an associated construction using initial momenta as follows. Let $P_k:=(\sum_{i=1}^km_i,\sum_{i=1}^km_iv_i)$, $P_0:=(0,0)$, and let $L_k$ be the line segment in $\RR^2$ connecting $P_{k-1},P_k$, for $1\leq k\leq n$. We may view $\Gamma(f):=\bigcup_{k=1}^nL_k$ as the graph of a continuous, piecewise linear function $f:[0,\sum_{i=1}^nm_i]\to\RR$, $f(0)=0$.

Recall that the convex envelope of $f$ is defined by 
\[conv(f):=\sup\{g\leq f:\, g\, \text{ is convex}\}.\]
The properties of $f$ imply $r:=conv(f):[0,\sum_{i=1}^nm_i]\to\RR$ is also a continuous, piecewise linear function; let $\Gamma(r)$ denote its graph and $(R_i)$ its linear components of constant slope, ordered by increasing $x$-coordinate of its left endpoint. Observe that $(R_i)$ has monotonically increasing slope due to the convexity of $r$. Together, we refer to $\Gamma(f)\cup\Gamma(r)$ as the \textit{cumulative momentum diagram} of our system. 

For each $R_i$, there are points $(P_{i_k})$, with $P_{i_k}\in R_i$, ordered by increasing $x$-coordinate. A \textit{polygon} $T$ in our diagram refers to either 
\begin{enumerate}[(a)]
    \item The simple polygon bounded by $\Gamma(f)$ and the portion of $R_i$ joining $P_{i_k}, P_{i_{k+1}}$, for $i_{k+1}\neq i_k+1$; or
    \item The line segment $L_{i_{k+1}}$ when $i_{k+1}=i_k+1$ (the ``degenerate'' case).
\end{enumerate}
We can now decompose our diagram into a sequence of polygons $(T_i)_{i=1}^\ell$, ordered using the $x$-coordinate of the vertex (or endpoint) with smallest $x$-coordinate in each polygon. This further induces a partition on our system $(a_i)_{i=1}^n$, where $a_i,a_j\in T_k$ if and only if $L_i,L_j$ belong to an edge of $T_k$. In an abuse of language, the \textit{slope of} $T$ refers to the slope of the linear component $R$ that bounds $T$. See Figure \ref{fig:sub2} for an example diagram, and note $T_3$ is a line segment that counts as a polygon by our rules above.

\subsection{Statement of theorem}
Our main result states that given a system $(a_i)_{i=1}^n$, there is a bijection between ordered clusters $A_i$ and ordered polygons, with $A_i=T_i$.

\begin{thm}\label{main}
Let $(a_i)_{i=1}^n$ be a system whose cumulative momentum diagram decomposes into polygons $(T_i)_{i=1}^\ell$. Then the system partitions into clusters $(A_i)_{i=1}^\ell$ with $A_j=T_j$, and the velocity of $A_j$ is given by the slope of $T_j$.
\end{thm}

\section{Proof of main result}

The key step in showing our main result is the ability to consider polygons in isolation, where particles outside a given polygon never interact with particles within. This is the statement of the following ``cutting'' lemma, which utilizes the monotonicity of $(R_i)$:

\begin{lem}\label{key}
Particles in $T_i$ never collide with particles in $T_\ell$ for $i\neq \ell$.
\end{lem}
\begin{proof}
By our observation following Property \ref{stick}, it suffices to show that $\dot\gamma_{j}(t+)\leq \dot\gamma_{j+1}(t+)$ for all $t\geq 0$, where $j$ is the index of the particle $a_j\in T_i$ with largest $x$-coordinate. Observe that until $a_j,a_{j+1}$ collide, their respective velocities at any time are given by 
\begin{align}
    \dot\gamma_{j}(t+)&=\frac{\sum_{k=0}^{n_1-1}m_{j-k}v_{j-k}}{\sum_{k=0}^{n_1-1}m_{j-k}}\label{left}\\ 
    \dot\gamma_{j+1}(t+)&=\frac{\sum_{k=1}^{n_2}m_{j+k}v_{j+k}}{\sum_{k=0}^{n_2}m_{j+k}} \label{right}
\end{align}
for $n_1,n_2\geq 1$, according to \eqref{velo}. The right hand side of \eqref{left} is the slope of the line passing through $P_{j-n_1}$ and $P_{j}$, which is sharply bounded above by $R$, the slope of $T_i$; this follows from $P_j\in\Gamma(r)$ and the convexity of $r$. Similarly, the right hand side of \eqref{right} is the slope of the line passing through $P_{j}$ and $P_{j+n_2}$, which is sharply bounded below by $\tilde R$, the slope of $T_{i+1}$ (see Figure \ref{fig:pf}). Due to monotonicity of polygon slopes, we know $R\leq \tilde R$, and therefore $\dot\gamma_{j}(t+)\leq \dot\gamma_{j+1}(t+)$ as needed.
\end{proof}

\begin{figure}
\centering
\begin{tikzpicture}
\draw (-5,3)node[circle,fill,inner sep=1.5pt]{}--(-4, 4)node[circle,fill, blue, inner sep=1.5pt]{}--(-3.5, 3.25)--(-2.5, 4.5)--(0,0)node[circle,blue,fill,inner sep=1.5pt]{}--(2, 3)--(2.5, 1.5)node[circle,fill,blue,inner sep=1.5pt]{}--(3,0)node[circle,fill,inner sep=1.5pt]{};    
\draw[dashed](-4,4)node[above,blue]{$P_{j-n_1}$}--(0,0)node[below,blue]{$P_{j}$}--(2.5, 1.5)node[right,blue]{$P_{j+n_2}$};
\draw[red](-5,3)--(0,0)--(3,0);
\end{tikzpicture}
\caption{}\label{fig:pf}
\vspace{10pt}
\end{figure}

\noindent\textbf{Proof of Theorem: } We proceed by induction. Suppose our system partitions into one polygon $T$ with slope $R=\frac{\sum_{k=1}^nm_kv_k}{\sum_{k=1}^nm_k}$. Let $(A_i)_{i=1}^N$ be the clusters  of our system, and denote the velocity of $A_i$ by $W_i$. If $J$ is the maximal collection of indices of particles belonging to $A_i$, let $M_i:=\sum_{k\in J}m_k$. 

Because these clusters do not collide, we have $W_i\leq W_{i+1}$. The velocity $W_1$ is given by  \eqref{right} with $j=0$, which by the reasoning in the proof of Lemma \ref{key} gives $W_1\geq R$. Combined with conservation of momentum, this yields
\[\sum_{k=1}^n m_kv_k=\sum_{j=1}^NM_jW_j\geq W_1\sum_{j=1}^NM_j\geq R\sum_{k=1}^nm_k=\sum_{k=1}^n m_kv_k.\]
It follows immediately that $W_1=R$ and $N=1$.

Assume our theorem holds for $\ell-1$ polygons, and consider a system partitioned into $\ell$ polygons. Let $a_j$ be the particle with largest $x$-coordinate belonging to the $T_{\ell-1}$. By Lemma \ref{key}, particles in $\bigcup_{i=1}^{\ell-1}T_i$ do not collide with particles in $T_{\ell}$. Thus, the first $\ell-1$ polygons may be considered in isolation from the rest of the system, so by our inductive hypothesis, $(a_i)_{i=1}^j$ partitions into $\ell-1$ clusters with $A_j=T_j$. We are then left with the particles in $T_\ell$, which upon applying our base case reasoning corresponds to the single cluster $A_\ell$. \qedsymbol

\section{Application to a toy problem}
Consider the following problem: Given a system of $n$ particles of unit mass with velocities taking values in $\{\pm 1\}$, we want to determine the probability of one cluster forming. By looking at the associated cumulative momentum diagrams, this is equivalent to summing over the probabilities that a simple symmetric random walk 
(SSRW) on $\mathbb{Z}$ starting at $0$ and ending at $c\in\mathcal{V}$ stays strictly above $\frac{c}{n}j$ for each step $0<j<n$, where $\mathcal{V}:=\{-n\leq c\leq n: c=n-2k, k\in\ZZ\}$ is the range of possible total velocities of the $n$ particles.

This type of problem appears in combinatorics in relation to the classic ballot problem and counting North-East lattice paths -- or \textit{staircase walks} -- on $\ZZ^2$ (see Chapters 1, 2 in \cite{10.3138/j.ctvfrxdw5}). To see this, observe that each SSRW satisfying our needs above corresponds to a lattice path on $\ZZ^2$ starting at the origin and ending at $(x,y)=((n-c)/2,(n+c)/2)$ that stays strictly above the line of slope $y=x((n+c)/(n-c))$; here, each unit step in the $+x$ (resp. $+y$)-direction corresponds to a $-1$ (resp. $+1$) step in our SSRW.

\subsection{Young diagrams and Narayana's path counting}
In order to approach the situation above, we briefly recall Young diagrams and a useful path counting formula. Given $\lambda=(\lambda_1,\ldots, \lambda_k)\in\ZZ_+^k$ with $\lambda_i\geq\lambda_{i+1}$ the \textit{Young diagram of shape $\lambda$} comprises left justified rows with row $i$ consisting of $\lambda_i$ unit boxes.

Let us introduce the notation 
\[{\binom{n}{k}}_+:=\begin{cases}
    {\binom{n}{k}}& \text{\, if\, }k\geq 0;\\
    0 & \text{\, if\, } k<0.\\
\end{cases}\]

\begin{figure}
\begin{tikzpicture}

    \draw[very thin] (0,0)node[circle,black,fill,inner sep=1.5pt]{}--(0,5)--(4,5)node[circle,blue,fill,inner sep=1.5pt]{}--(4,4)--(2,4)--(2,3)--(1,3)--(1,2)--(0,2);
    \draw[very thin] (1,3)--(1,5);    
    \draw[very thin] (0,4)--(2,4);
    \draw[very thin] (0,3)--(1,3);
    \draw[very thin] (2,4)--(2,5);
    \draw[very thin] (3,4)--(3,5);
    \draw (0,1)node[circle,black,fill,inner sep=.5pt]{};
    \draw[dashed, red] (0,0)--(0,2)--(1,2)--(1,4)--(2,4)--(2,5)--(4,5);
\end{tikzpicture}
    \centering
    \caption{Example of a North-East lattice path on the Young diagram of shape $\lambda=(4,2,1,0,0)$.}
    \label{fig:expath}
    \vspace{10pt}
\end{figure}

We know the following formula:
\begin{thm}[Path Counting, \cite{ciucu2016short} \cite{BURO_1965__6__9_0}]
The number $N(\lambda)$ of North-East lattice paths joining the southwest corner to the northeast corner of a Young diagram of shape $\lambda$ is given by
\begin{align}\label{narpath}
    N(\lambda)=\det \left({\binom{\lambda_j+1}{j-i+1}}_+\right)_{1\leq i,j\leq k}
\end{align}
\end{thm}
Observe that the matrix that we are taking the determinant of in \eqref{narpath} has all ones on the main subdiagonal and zeros all below that. Also, if we extend the notion of Young diagram to tuples $\lambda\in\ZZ_{\geq 0}^k$ with the zero entries corresponding to rows with left-aligned single points instead of boxes, equation  \eqref{narpath} still holds due to determinant properties of block lower-triangular matrices; in this case, $N(\lambda)$ counts the paths on the diagram given by the corresponding Young diagram for nonzero entries of $\lambda$ with the additional vertical line segment joining the single points (Figure \ref{fig:expath}). 

\subsection{Solution to the problem}

Observe that for $c=\pm n,\pm (n-2)$, the probabilities of one cluster forming are $0$ and $(1/2)^n$ respectively. For $c\in\mathcal{V}$ that are not those values, the number of $\ZZ^2$ lattice paths from the origin to $((n-c)/2,(n+c)/2)$  staying above $y=x((n+c)/(n-c))$ is given by $N(\lambda)$ for 
\[\lambda=\left(\left\lceil\left(\frac{n-c}{n+c}\right)\left(\frac{n+c}{2}-j\right)-1\right\rceil\right)_{j=1}^{ \frac{n+c}{2}-1}.\]
This $\lambda$ is the shape of the maximal Young diagram that fits inside  $[0,(n-c)/2]\times [0,(n+c)/2]$ and stays strictly above $y=x((n+c)/(n-c))$.
Plugging in $c=-n+2k$ for $2\leq k\leq n-2$, we find the probability of one cluster forming to be
\[(1/2)^{n-1}+\frac{1}{2^n}\sum_{k=1}^{n-1}\det \left({\binom{\left\lceil\left(\frac{n-k}{k}\right)\left(k-j\right)-1\right\rceil+1}{j-i+1}}_+\right)_{1\leq i,j\leq k-1}.\]
Noting in this case that the number of permissible lattice paths from the origin to $(p,q)$ is the same as that for $(q,p)$ (see Figure \ref{fig:symmetry}), we may rewrite the above slightly and conclude that the probability of one cluster forming for $n>2$ even (resp. odd) is:
\begin{align*}
    \frac{1}{2^{n-1}} +\frac{1}{2^{n}}\det \left({\binom{\left\lceil (n/2)-j-1\right\rceil+1}{j-i+1}}_+\right)_{1\leq i,j\leq (n/2)-1} \\
    +\frac{1}{2^{n-1}}\sum_{k=2}^{n/2-1}\det \left({\binom{\left\lceil\left(\frac{n-k}{k}\right)\left(k-j\right)-1\right\rceil+1}{j-i+1}}_+\right)_{1\leq i,j\leq k-1}
\end{align*}
    \[\left[resp.\,\hspace{5pt} \frac{1}{2^{n-1}}    +\frac{1}{2^{n-1}}\sum_{k=2}^{\frac{n-1}{2}-1}\det \left({\binom{\left\lceil\left(\frac{n-k}{k}\right)\left(k-j\right)-1\right\rceil+1}{j-i+1}}_+\right)_{1\leq i,j\leq k-1}\right]\]
\begin{figure}
\begin{tikzpicture}

    \draw[thick,-] (0,0) -- (6.5,0) ;
    \draw[thick,-] (0,0) -- (0,6.5) ;
    \foreach \x in {0,1,2,3,4,5,6}
    \draw (\x cm,1pt) -- (\x cm,-1pt) node[anchor=north] {$\x$};
    \foreach \y in {1,2,3,4,5,6}
    \draw (1pt,\y cm) -- (-1pt,\y cm)node[anchor=east] {$\y$};
    \foreach \i in {0,...,6} {
    \draw [very thin,gray] (\i,0) -- (\i,6) ;
    }
    \foreach \i in {0,...,6} {
    \draw [very thin,gray] (0,\i) -- (6,\i);
    }
    \draw (0,0)node[circle,fill,inner sep=1.5pt]{}-- (3,5)node[circle,blue,fill,inner sep=1.5pt]{}node[anchor=south] {$(3,5)$};
    \draw (0,0) -- (5,3)node[circle,blue,fill,inner sep=1.5pt]{}node[anchor=south] {$(5,3)$};
    \draw [red] (0,2)--(0,5)--(2,5)--(2,4)--(1,4)--(1,2)--(0,2);
    \draw [red] (0,3)--(1,3);
    \draw [red] (0,4)--(1,4);
    \draw [red] (1,4)--(1,5);
    
    \draw [dashed] (0,1)--(0,3)--(3,3)--(3,2)--(1,2)--(1,1)--(0,1);
    \draw [dashed] (0,2)--(1,2);
    \draw [dashed] (1,2)--(1,3);
    \draw [dashed] (2,2)--(2,3);
    
\end{tikzpicture}
    \centering
    \caption{}
    \label{fig:symmetry}
    \vspace{10pt}
\end{figure}

\subsection{Some remarks on asymptotics}
We conclude this section with a brief discussion on the asymptotics of the probability that one cluster forms from our unit mass and speed system as the number of particles $n\to \infty$. Our approach to find the above formula does not appear to be conducive in extracting asymptotics due to the complicated determinant terms arising from path-counting. Some relevant asympototics for path-counting below a certain class of lines of rational slope can be found in \cite{banderier2019kernel}, though the techniques there do not seem to easily generalize.

Note that for a SSRW $(S_k)_{k=1}^{2n}$ starting at $0$, the probability $S_k$ remains strictly positive at each step given the walk ends at $0$ is
\[\PP\left(\bigcap_{1<k<2n}(S_k>0)|S_{2n}=0\right)=\left(\frac{1}{2}\right)^{2n}C_{n-1}=\left(\frac{1}{2}\right)^n\frac{1}{n}\binom{2n-2}{n-1}\sim \frac{e}{2\sqrt{\pi}}\left(\frac{1}{\sqrt{n}}\right),\]
where $C_k$ denotes the $k$-th Catalan number. Applying Stirling's formula once again to find
\[\PP(S_{2n}=0)\sim\frac{1}{\sqrt{\pi}}\left(\frac{1}{\sqrt{n}}\right),\]
we get
\[\PP\left(\bigcap_{1<k<2n}(S_k>0),S_{2n}=0\right)\sim \frac{e}{2\pi}\left(\frac{1}{n}\right).\]

Heuristically, because we expect most of our walks to satisfy $S_{2n}=c$ for small values of $c$ with high probability, we conjecture that summing over these possibilities leads to a similar asymptotic of $\sim C/n$.
\section{Computational applications}
We seek an algorithm that takes in a sequence of original particles as a list of $(M_i,V_i)$ pairs, and outputs a similar list of clusters after all collisions occur. An outline for one implementation of this code utilizing recursion is as follows:
\begin{enumerate}
    \item Begin with an initial cumulative momentum diagram for the entire sequence. For each point on the diagram in which the slope between it and its previous point is less than or equal to the slope between it and the following point, add said point to the list. End the list with the final point of the diagram.
    \item Recursively repeat the previous step with just these points acting as a revised, shortened momentum diagram. The terminating condition for said recursion occurs when this process no longer produces a shortened diagram. Take this constant result as the final diagram.
    \item The line segments between consecutive points in this final diagram represent the final polygons in the diagram--the slopes of the segments are the velocities of the final particles while the horizontal magnitude of the line is the mass in the unit case.
\end{enumerate}

This algorithm results in a method of determining the final state of coalescence of a sequence in a much more streamlined manner than the brute-force method involving purely the physical calculations. This proves to be especially helpful in the construction of probability density functions for the state of an $n$-length system as time goes to infinity. In this type of problem, we seek to find the probabilities of there being $k\leq n$ final clusters in an $n$-length unit-mass system when assigned random velocities of $\pm 1$. For low values of $n$ one can fairly easily compute these functions by hand, as each $n$-length sequence encompasses a total of $2^{n}$ possible configurations. However, to analyze enough of the lower $n$-length sequences, this exponential growth in the size of the input requires an algorithm that can save as much time as possible when compared to that using purely physical calculations. Following trials on randomly generated sets of sequences, algorithms based on analyzing momentum diagrams with the polygon theorem as as a basis achieves this goal; and, in addition, run with more economical memory usage as well. An example of Python code utilizing this algorithm to find the final state of an arbitrary sequence of particles is below.
\begin{verbatim}

# Takes two [x,y] points as input
# Returns slope between both points
def slope(first, second):
    rise = second[1] - first[1]
    run = second[0] - first[0]
    return rise/run

# Takes a list of all [x,y] points of a momentum diagram as input
# Returns only the points used as nodes in said momentum diagram
def polygonAnalysis(seq):
    if len(seq) < 2:
        return seq
       
    else:
        mins = [seq[0]]
        curr= seq[1]
        currIndex = 1
        while currIndex < len(seq) - 1:
            prev = seq[currIndex - 1]
            next = seq[currIndex + 1]
            local_Min = False
            if slope(prev, curr) <= slope(curr, next):
                local_Min = True
            if local_Min:
                mins.append(seq[currIndex])
            currIndex = currIndex + 1
            curr = seq[currIndex]

        mins.append(seq[len(seq) - 1])
        # Algorithm terminates when no further reductions occur
        if mins == seq:
            return mins
        else:
            return polygonAnalysis(mins)

# Takes an original sequence of particles in the form [mass, velocity] as input 
# Returns the final list of particles after all collisions ordered by position.
def findFinalTree(origSequence):
    origParticles = []
    index = 0
    for item in origSequence:
        origParticles.append(item)
        index = index + 1
    
    numStart = 0
    numDown = 0

    curr_x = numStart
    curr_y = numDown

    rldList = [[curr_x, curr_y]]
    for item in origParticles:
        curr_x = curr_x + item[0]
        curr_y = curr_y + (item[1] * item[0])
        rldList.append([curr_x, curr_y])

    middleParticles = polygonAnalysis(rldList) # Recursive step
    currMIndex = 0
    finalParticles = []
    while currMIndex < len(middleParticles) - 1:
        firstParticle = middleParticles[currMIndex]
        secondParticle = middleParticles[currMIndex + 1]
        mass = secondParticle[0] - firstParticle[0]
        finalParticles.append([mass, slope(firstParticle, secondParticle)])
        currMIndex = currMIndex + 1

    return [finalParticles]

# Runs program, prints to terminal list of final particles after all collisions    
origSequence = ANY_ARBITRARY_PARTICLE_LIST
print(findFinalTree(origSequence))
    
\end{verbatim}

\printbibliography
\end{document}

%% file: defs.tex
\def\det{\mbox{det }}

\newtheorem{lem}{Lemma}
%
%

\def\PP{{\mathbb P}}

\def\RR{{\mathbb R}}

\def\ZZ{{\mathbb Z}}

\def\1{{\mathbbm 1}}

%
%

%
%

%
%

%
%





\def\det{\operatorname{det}}



